\documentclass{amsart}
\usepackage{amsfonts,amscd}

\newtheorem{theorem}{Theorem}[section]
\newtheorem{lemma}[theorem]{Lemma}
\newtheorem{corollary}[theorem]{Corollary}
\newtheorem{proposition}[theorem]{Proposition}
\theoremstyle{remark}

\theoremstyle{definition}
\newtheorem{definition}[theorem]{Definition}

\numberwithin{equation}{section}
\makeatother

\newcommand{\fc}{\frac{1}{2}\mathfrak{F}_{\mathbb{C}}}

\newcommand{\half}{\frac{1}{2}}
\newcommand{\norm}[1]{\left \|#1\right\|}

\DeclareMathOperator{\newc}{\mathfrak{c}}
\DeclareMathOperator{\re}{\mathfrak{r}}

\DeclareMathOperator{\Cdb}{{\mathbb C}}
\DeclareMathOperator{\Rdb}{{\mathbb R}}
\DeclareMathOperator{\Ddb}{{\mathbb D}}

\DeclareMathOperator{\Ndb}{{\mathbb N}}

\begin{document}

\title[Positivity and Roots in Operator Algebras]{On Positivity and Roots in Operator Algebras}

\author{Clifford A. Bearden}
\address{Department of Mathematics \\
 University of Houston \\
 Houston, TX 
77204-3008}
\email{cabearde@math.uh.edu}

\author{David P. Blecher}
\address{Department of Mathematics s \\
 University of Houston \\
 Houston, TX
77204-3008}
\email{dblecher@math.uh.edu}

\author{Sonia Sharma}
\address{Department of Mathematics \\
SUNY Cortland \\
 Cortland, NY 13045}
 \email{sonia.sharma@cortland.edu}

\thanks{The first two authors were supported by a grant from the NSF.
Revision of March 2014}


\begin{abstract}   In earlier papers the second author and Charles Read have introduced and studied a new 
notion of positivity for operator algebras, with an eye to
extending certain $C^*$-algebraic results and theories to more 
general algebras.   The present paper consists of complements to some facts in the just mentioned papers,
concerning this notion of positivity.     For example we prove a result on the numerical range of products of the roots of commuting operators with numerical range in a sector.   
  \end{abstract}

\maketitle

\section{Introduction}

An {\em operator algebra} is a closed subalgebra $A$ of $B(H)$, for a
Hilbert space $H$.  We will be working over the complex field
always, so $H$ is a complex Hilbert space. 
In operator theory and in the theory of selfadjoint  operator algebras ($C^*$-algebras and von Neumann algebras), it is hard to overestimate
the importance of the role played by 
positive elements and their roots.      
In earlier papers \cite{BR1, BR2, BRIII,Read} the second author and Charles Read have shown that many of these crucial positivity ideas
carry over to 
more general operator algebras (see also e.g.\ \cite{Bnew,BN2}).  
These authors introduce and study a new 
notion of positivity for operator algebras, even in
 algebras with no nonzero 
positive elements in the usual sense.  This is done with an eye to
extending certain $C^*$-algebraic results and theories to more 
general algebras.   A central role is played by the set ${\mathfrak F}_A = \{ a \in A : 
\Vert 1 - a \Vert \leq 1 \}$,  and the cones 
$${\mathfrak c}_A = 
\Rdb_+ \, {\mathfrak F}_A, \; \; \text{and} \; \; {\mathfrak r}_A = \overline{{\mathfrak c}_A} = \{ a \in A : a + a^* \geq 0 \}.$$
Elements of these sets and their roots play the role in many situations of positive elements
in a $C^*$-algebra.  
The present paper consists of complements to some facts in the just mentioned papers,
concerning this `positivity'.      
For example, in Section 2, we characterize `real completely positive maps' relative to
our cone ${\mathfrak r}_A$.  In Section 3 we clarify a point about the support projection 
$s(x)$ from \cite{BR1}.  In Section 4 we prove some interesting and surprising facts about `roots' of 
operators in our `positive cone'.   These results are used in \cite{BRIII}, and should be useful elsewhere.   For example we prove that the product of suitable roots of commuting 
operators in ${\mathfrak r}_A$ (resp.\ in ${\mathfrak F}_A$) is again 
in ${\mathfrak r}_A$ (resp.\ in ${\mathfrak F}_A$).     There is an  extensive literature on accretive products of matrices and operators
(see e.g.\ \cite{HJ,GR} for some discussion of the difficulties and 
basic results here); such product formulae are rare and have extremely important applications.  
Our result in the accretive case will not be surprising to experts on sectorial operators, however it does not appear to be in the literature.  
 It is related to theorems of R. Bouldin, J. Holbrook, T. Kato,  J. P. Williams, and others (see e.g. \cite{GR} for references).

We now state our notation, and some facts. We refer the reader to
\cite{BLM,BR1,BR2} for additional background on operator algebras, and for
some of the details and notation below.  
We reserve the letter $H$ for a complex Hilbert space,
usually the Hilbert space
on which our operator algebra is acting, or is completely isometrically represented.  
We write $X_+$ for the positive operators (in the usual sense) that happen to
belong to a subset $X$ of $B(H)$ or of a $C^*$-algebra. 
We write oa$(x)$
for the operator algebra generated by $x$ in $A$, namely the smallest closed
subalgebra of $A$ containing $x$.      

The second dual $A^{**}$ is also an operator algebra with its (unique)
Arens product.  This is also the product inherited from the von Neumann
algebra $B^{**}$ if
$A$ is a subalgebra of a $C^*$-algebra $B$.  (We use the same symbol $*$ for the Banach dual
and for the involution or adjoint operator, the reader will have to determine which
is meant from the context.)   
Note that
$A$ has a contractive approximate
identity (cai) iff $A^{**}$ has an identity $1_{A^{**}}$ of norm $1$.
In this case we say that
$A$ is {\em approximately unital}.

We recall that by a theorem due to  Ralf Meyer,
every operator algebra $A$ has a  unitization $A^1$
which is unique up
to completely isometric homomorphism (see
 \cite[Section 2.1]{BLM}). Below $1$ always refers to
the identity of $A^1$ if $A$ has no identity. If $A$ is a nonunital
operator algebra represented (completely) isometrically on a Hilbert
space $H$ then one may identify $A^1$ with $A + \Cdb I_H$.  
For an operator algebra, not necessarily approximately  unital,
we recall that ${\mathfrak r}_A$ is the cone of  elements with 
positive `real part'.  As we just said, $A^1$ is uniquely defined, and can be viewed
as $A + \Cdb I_H$.  Hence 
$A^1 + (A^1)^*$ is also uniquely defined, by e.g.\ 1.3.7 in \cite{BLM}.
We define $A + A^*$ to be the obvious subspace
of $A^1 + (A^1)^*$.  This is well defined independently
of the particular Hilbert space $H$ on which $A$
is represented, as shown at the start of Section 3 in \cite{BR2}.
 Thus a statement such as
$a + b^* \geq 0$ makes sense whenever $a, b \in A$, and is independent of the particular $H$ on which $A$
is represented, and in particular the same is
true for ${\mathfrak r}_A = \{ a \in A : a + a^* \geq 0 \}$.   Elements in ${\mathfrak r}_A$, that is elements in $A$ with
${\rm Re}(x) = x + x^* \geq 0$,  will sometimes be called
{\em accretive} (although this term is often used in the
literature for
a slightly different notion).

We recall that $\frac{1}{2} {\mathfrak F}_A = \{ a \in A : \Vert 1 - 2 a \Vert \leq 1 \}$.   Here $1$ is 
the identity of the unitization $A^1$ if $A$ is nonunital.  
It is easy to see
that $x \in {\mathfrak c}_A = \Rdb_+ {\mathfrak F}_A$ iff there is a positive
constant $C$ with $x^* x \leq C(x+x^*)$.
These sets (and ${\mathfrak r}_A$) are studied in earlier work as analogues of the positive cone
of a $C^*$-algebra (particularly in \cite{BR1, BR2, BRIII, BN1, BN2}). 
We showed in \cite[Section 3]{BR2} that ${\mathfrak r}_A = \overline{{\mathfrak c}_A}$,
where ${\mathfrak c}_A = \Rdb_+ \, {\mathfrak F}_A$.  It is clear that $\overline{{\mathfrak c}_{A^1} \cap A} = {\mathfrak r}_{A^1} 
\cap A = {\mathfrak r}_A$.

By the {\em numerical range} $W(x)$ of an element $x$, we will mean the set of values $\varphi(x)$ for states $\varphi,$
while the literature we quote usually uses the one defined by vector states on $B(H)$.  However since the former range is the 
closure of the latter, as is well known, this will cause no difficulties.  
For any operator $T \in B(H)$ whose numerical range does not include strictly negative
numbers, and for any $\alpha \in [0,1]$, there is a well-defined `principal' root $T^\alpha$,
which obeys the usual law $T^\alpha T^\beta = T^{\alpha + \beta}$ if $\alpha + \beta \leq 1$
(see e.g.\ \cite{MP,LRS}).   Write $S_\psi$ for  the
 sector $ \{ r e^{i \theta} : 0 \leq r , \, \text{and} \,  
-\psi \leq \theta \leq \psi \}$ where $0 \leq \psi < \pi$.
Then $T \mapsto T^\alpha$ is continuous on operators with numerical range
in $S_\psi$.   Our operators $T$ will be accretive (that is,
 $\psi \leq \frac{\pi}{2}$), 
and then these                 
powers obey the usual laws such as  $T^\alpha T^\beta = T^{\alpha + \beta}$
for all $\alpha, \beta > 0$,
$(T^\alpha)^\beta
= T^{\alpha \beta}$ for $\alpha \in (0,1]$ and any $\beta > 0$, 
and $(T^*)^\alpha  = (T^\alpha)^*$.     
If  $n \in \Ndb$ then $T^{\frac{1}{n}}$ is the unique $n$th root of $T$ with numerical 
range in $S_{\frac{\pi}{2n}}$, for any $\alpha \geq 0$.   See e.g.\ \cite[Chapter IV, Section 5]{NF} and \cite{Haase}
for all of these facts.   We also have 
 $(cT)^\alpha  = c^\alpha T^\alpha$ for positive scalars $c$.     Also $\alpha \mapsto T^{\alpha}$ 
is continuous on $(0,\infty)$, for $T$ accretive.  For $x \in {\mathfrak r}_A$ we have $x^\alpha \in {\rm oa}(x)$ if  
$\alpha > 0$.  See \cite[Lemma 1.1]{BRIII} for these facts.  
In particular, ${\mathfrak r}_A$ is closed under taking roots.

\section{Positivity and Real Complete Positivity}

In \cite[Section 8]{BR1}, the second author and Read  defined a class of linear 
maps called OCP  or {\em operator completely positive} (the precise definition is given below), 
and proved  an extension and  Stinespring dilation  theorem for them.   
In particular it was shown  that the linear
$B(H)$-valued OCP maps on a unital operator space or approximately unital operator algebra
$A$ are the restrictions to $A$ of linear completely positive maps in the usual sense on an enveloping $C^*$-algebra.   
Here we do the same for maps respecting the `cone' of elements with positive real part.  
If $A$ is a unital operator space in the sense of e.g.\  1.3.1 in \cite{BLM}, and \cite{BN0} (and its sequel by the same authors),
notice that  ${\mathfrak r}_A = \{ x \in A : x + x^* \geq 0 \}$ also makes sense, and is closed.  This is because 
the operator system $A + A^*$ is well defined independently of the representation of  $A$ as a unital operator space on a Hilbert space,
by e.g.\ 1.3.7 in \cite{BLM}.  Clearly $\re_A$ spans $A$ in this case, since $1 \in \re_A$,  and for any $x \in A$ we have 
$x + t 1 \in
\re_A$ for large enough $t$.  If $A \subset B$ is a unital containment of unital operator spaces, then $\re_A \subset \re_B$.  Finally,
$\re_A = \overline{\Rdb_+ {\mathfrak F}_A}$.   To see this notice that $\Rdb_+ {\mathfrak F}_A \subset \re_A$ as in the operator algebra 
case.  For the other direction, if $x \in \re_A$ then $${\rm Re}(x + \frac{1}{n})
\; \geq \, \frac{1}{n} \, 1 \, \geq \; 
C \, (x + \frac{1}{n})^*  \, (x + \frac{1}{n})$$
for some positive
 constant $C$.
Hence $x + \frac{1}{n} \in \Rdb_+ {\mathfrak F}_A$, and so $x \in \overline{\Rdb_+ {\mathfrak F}_A}$.

\begin{definition} \label{rcdef} A linear completely bounded map $u: A\to B$ between operator algebras (or between unital operator spaces) is {\it real completely positive} (RCP) if $u(x)+u(x)^*\geq 0$ whenever $x\in A$ with $x+x^*\geq 0$, and similarly for $x\in M_n(A)$ for all $n\in \Ndb$. In other words, $u_n(\re_{M_n(A)})\subseteq \re_{M_n(B)}$ for all $n\in \Ndb$.
\end{definition}  

It is clear from properties of $\re_A$ mentioned earlier, that restrictions of RCP maps to subalgebras, or to unital operator subspaces, are again RCP.

The OCP maps were defined in \cite{BR1} similarly to Definition \ref{rcdef}, but requiring the existence of a positive constant $C$ with 
$u_n({\mathfrak F}_{M_n(A)}) \subset C \,  {\mathfrak F}_{M_n(B)}$ for
every $n \in \Ndb$.   If  $A, B$ are operator algebras or unital operator spaces, and if $T: A\to B$ is OCP then $T$ is RCP.  This follows
easily from the definitions, and  the fact that $\re_A=\overline{\newc_A}$.
Indeed  $T(\overline{\newc_A})\subseteq \overline{T(\newc_A)}\subseteq \overline{\newc_B}$ if $T$ is OCP, and a  similar argument applies at each matrix level.  The following  is also clear:

\begin{lemma}    Restrictions of a linear completely positive map from a $C^*$-algebra into $B(H)$, to a subalgebra
or unital subspace, are RCP.
\end{lemma}

\begin{lemma} \label{ikhuh} 
If $A$ is a $C^*$-algebra or operator system, then $x\in A_+$ if and only if $zx\in \re_A$ for all $z\in \fc$.   This is equivalent to: $zx\in \re_A$ for all $z\in \re_{\Cdb}$.
  \end{lemma}

\begin{proof}  We just prove the first `iff', from which the second equivalence is clear.

 ($\Rightarrow$) \ This is obvious.  

($\Leftarrow$) \ (C.f.\ \cite[Lemma 8.5]{BR1}.) \ If $zx\in \re_A$, then by definition, ${\rm Re}(zx)\geq 0$, and hence ${\rm Re}(z \langle x\zeta, \zeta \rangle) \geq 0$, for all $\zeta\in H$ and for all $z\in \fc$. By calculus this implies that
$\langle x\zeta, \zeta \rangle \geq 0$ for all $\zeta\in H$.
So $x\in A_+$.  
\end{proof}

\begin{theorem} \label{27}
If $T: A\to B$ is a linear map between $C^*$-algebras or operator systems then $T$ is completely positive iff $T$ is RCP.
\end{theorem}  \begin{proof}   This is similar to the proof in \cite{BR1}, but for convenience we give the details.
Clearly any completely positive map is RCP.  
Conversely, if $T : A \to B$ is RCP and  $x \in {\rm Ball}(A)_+$ and $z \in {\mathfrak F}_{\Cdb}$
then $zx  \in {\mathfrak r}_{A}$
by Lemma \ref{ikhuh}.  Thus $zT(x) = T(zx)  \in {\mathfrak r}_{B}$,
and so $T(x) \geq 0$ by Lemma \ref{ikhuh}.  A similar argument applies to matrices.
\end{proof}

\begin{theorem}  \label{db8}
If $T: A\to B(H)$ is a linear RCP map on a unital operator space $A$, then the canonical extension $\tilde{T}: A+A^*\to B(H): x+y^*\mapsto T(x)+T(y)^*$ is well defined and completely positive.
\end{theorem}
\begin{proof}
Let $T : A \to B(H)$ be RCP.   Then $T$ restricted to $\Delta(A)
= A \cap A^*$
is RCP, and so it is completely positive and selfadjoint by Theorem \ref{27}.  Define $\tilde{T}(a + b^*) = T(a) + 
T(b)^*$
for $a, b \in A$.  
To see that   $\tilde{T}$ is well defined, suppose $a + b^* = x+ y^*$,  for $a, b, x, y \in A$.  Then $a-x = (y-b)^* \in\Delta(A)$, and 
so $$T(a-x) = T((y-b)^*) = (T(y) - T(b))^* .$$  Thus, $T$ is well defined.
If $z = a + b^*$ is positive (usual sense), then $$z = z^* = b + a^* = \frac{1}{2} (a + 
b^* + b + a^*)
= \frac{1}{2} (a + b) + (\frac{1}{2} (a + b))^*,$$ and $\frac{1}{2} (a + b) \in A$.  Since $T$ is RCP,
we have
 $$\tilde{T}(z) = T(\frac{1}{2}(a + b)) + T(\frac{1}{2}(a + b))^* \geq 0.$$ 
  So $\tilde{T}$ is positive, and a   
similar argument at the matrix levels shows that $\tilde{T}$ is completely positive.
\end{proof}

\begin{theorem}[Extension and Stinespring Dilation for RCP Maps] 
\label{db9} If $T:A\to B(H)$ is a linear map on a unital operator space or on an approximately unital operator algebra,  and if $B$ is a $C^*$-algebra containing $A$, then $T$ is RCP iff $T$ has a completely positive extension $\tilde{T}:B\to B(H)$. This is equivalent to being able to write $T$ as the restriction to $A$ of  $V^*\pi(\cdot)V$ for a $*$-representation $\pi: B\to B(K)$, and an operator  $V: H\to K$. Moreover, this can be done with $\norm{T}=\norm{T}_{cb}=\norm{V}^2$, and this equals $\norm{T(1)}$ if $A$ is unital. 
\end{theorem}
\begin{proof}   The structure of this proof follows 
the argument in \cite[Theorem 8.9]{BR1}, but using 
the results we have established above, in particular Theorem \ref{db8}, in place of their ${\mathfrak F}_A$ variants
from \cite[Section 8]{BR1}.  Thus if 
$T: A\to B(H)$ is an RCP map on an approximately unital operator algebra, 
let $\tilde{T}: A^{**}\to B(H)$ be the canonical weak$^*$ continuous extension.
Since $\overline{{\mathfrak r}_A}^{w*} = {\mathfrak r}_{A^{**}}$ (a fact
first proved in
 \cite[Section 3]{BR2}), we have $$\tilde{T}({\mathfrak r}_{A^{**}}) 
= \tilde{T}(\overline{{\mathfrak r}_A}^{w*}) \subset \overline{T({\mathfrak r}_A)}^{w*}
\subset {\mathfrak r}_{B(H)}.$$
Similarly at the matrix levels, so that $\tilde{T}: A^{**}\to B(H)$ is RCP
on the unital operator algebra $A^{**}$.  We  now follow the lines of the proof of 
\cite[Theorem 8.9]{BR1}, but using 
the results established above.
\end{proof}

\noindent  {\em Remark.}  The last result is connected (see e.g.\ \cite{BRIII})
 to the theory of real states of operator algebras.

\begin{corollary} \label{db10} 
Let $T:A\to B(H)$ be a linear map on a unital operator space or a (not necessarily approximately unital)
operator algebra $A$.  Then $T$ is  OCP iff it is RCP. 
\end{corollary}
\begin{proof}  If $A$ is a unital operator space or approximately unital  operator algebra then this
follows from  Theorem \ref{db9} and \cite[Theorem 8.9]{BR1}.  

If $A$ is a 
nonunital  operator algebra, let $A_H$ be the largest approximately unital subalgebra in $A$ as in \cite[Section 4]{BR2}, and let $S$
be the restriction of $T$ to $A_H$.
We have ${\mathfrak r}_A = {\mathfrak r}_{A_H}$ and ${\mathfrak F}_A = {\mathfrak F}_{A_H}$ by \cite[Section 4]{BR2}, so that
$T({\mathfrak r}_A) \subset {\mathfrak r}_{B(H)}$ iff 
$S({\mathfrak r}_{A_H}) \subset {\mathfrak r}_{B(H)}$.  Similarly for ${\mathfrak F}_A$ and ${\mathfrak F}_{A_H}$.
This is true at each matrix level too since $M_n(A_H) = M_n(A)_H$ by a lemma in \cite[Section 2]{BRIII}.  
Thus $T$ is RCP (resp.\ OCP) on $A$ iff 
$S$ is RCP (resp.\ OCP) on $A_H$.   This reduces the question to the case of approximately unital  operator algebras above.
\end{proof}

\section{The Support Projection} 

Note that if $x$ is an element of a subalgebra $A$ of $B(H)$ then there are two natural left support projections for $x$.  First there is
 the left support projection in $A^{**}$, namely the smallest 
projection $p$ in $A^{**}$ such that $px = x$ (assuming that there is such a $p$,
if there is not simply use the identity of the unitization).  Second, we have the left support projection
in $H$, the smallest projection $P$ in $B(H)$ such that  $Px = x$.   
Note that this is the projection onto $\overline{xH}$.
  If the left support projection in $A^{**}$, namely  $p$,
 is in the weak* closure 
of  $xAx$, and if $\pi : A^{**} \to B(H)$ is
 the natural weak*-continuous
homomorphism extending the inclusion map on $A$, then $\pi(p) = P$,
the left support projection
in $H$.    To see this, note that  $\pi(p) x = \pi(px) = \pi(x) = x$
in $B(H)$, so that $P \leq \pi(p)$.   If $x_t \to p$ weak* with $x_t \in xAx$, then $$P \pi(p) = \lim_t \, P x_t = 
\lim_t \, x_t = \pi(\lim_t \, x_t) = 
\pi(p),$$ so $\pi(p) \leq P$.  

Similarly  there are two natural right support projections, the second one being 
the projection onto $H \ominus {\rm Ker}(x)$.  If the left and right 
support projections of $x$ in $A^{**}$ coincide,  then we call this 
{\em the support projection of} $x$, written $s(x)$.  If this holds, and if $s(x) \in \overline{xAx}^{w*}$, then by the above we 
can also  conclude that the left and right 
support projections of $x$ in $B(H)$ coincide (and equal $\pi(s(x))$).  This will be the case for us below.

  The following is a generalization of \cite[Lemma 2.5]{BR1} (which
was the case when $x \in {\mathfrak F}_A$).  

\begin{proposition}   \label{supp}   For any operator algebra $A$,
if $x \in A$ with $x + x^* \geq 0$ and  $x \neq 0$,
then the left support projection of $x$ in $A^{**}$
 equals the right support projection, and equals $s(x (1+x)^{-1})$,
where the latter is the support projection
studied in {\rm \cite{BR1}}.   This also is the weak* limit of the net $(x^{\frac{1}{n}})$, and 
is an open projection in $A^{**}$
in the sense of  {\rm \cite{BHN}}.  If $A$ is a subalgebra of $B(H)$ then the left and right support projection of $x$ in $H$
are also equal.  
 \end{proposition}

\begin{proof}  The first part follows  the lines of the proof of \cite[Lemma 2.5]{BR1}.   For example,
 if $x \in A$ with $x + x^* \geq 0$, then the operator algebra oa$(x)$ 
generated by $x$ was shown in \cite[Section 3]{BR2} to have $(x^{\frac{1}{n}})$
as a bounded approximate identity.  The weak* limit of the latter is the support projection of $x$ by the proof of \cite[Lemma 2.5]{BR1}.  
To see that  the support projection equals $s(x (1+x)^{-1})$, simply note that $px = x$ iff
$p x (1+x)^{-1} = x (1+x)^{-1}$.  The last assertion  follows  from the considerations above the Proposition.  
 \end{proof}

This result has many consequences that are spelled out in \cite{BR2, BRIII}.

\section{Some Properties of Roots in an Operator Algebra}

We need a simple fact about the `bidisk algebra functional calculus' $f \mapsto f(S,T)$,
where $f$ is in the  bidisk algebra $A(\Ddb^2)$, and $S, T$ are commuting contraction operators.
This calculus is essentially the  two-variable von Neumann inequality 
resulting from Ando's dilation theorem (see e.g.\ 2.4.13 in \cite{BLM}).  We need the relationship between the bidisk  functional calculus, and the 
`disk algebra functional calculus'  $h \mapsto h(T)$ coming from the usual von Neumann inequality
(written as $u_T$ in 2.4.12  in \cite{BLM}).   There are no doubt more sophisticated variants in the literature
(see e.g.\ \cite{LLL}), however for the readers convenience we give a short proof.  

\begin{lemma}   \label{vnfc}
If $f \in A(\Ddb^2)$, the bidisk algebra, and $g, h \in A(\Ddb)$, with $\Vert g \Vert_{A(\Ddb)} \leq 1$
and $\Vert h \Vert_{A(\Ddb)} \leq 1$,
and if $S, T \in B(H)$ are commuting contractive operators,
then $f(g(S),h(T)) = (f \circ (g, h))(S,T)$.
 \end{lemma}

\begin{proof}  Let $\rho_{g,h} : A(\Ddb^2) \to A(\Ddb^2) : p \mapsto p(g(z),h(w))$.  Then $\rho_{g,h}$ is
a contractive homomorphism.  Let $\theta_{S,T} : A(\Ddb^2) \to  B(H) : p \mapsto p(S,T)$
be the bidisk algebra functional calculus, a contractive homomorphism, as is
$\theta_{g(S),h(T)}$.  Claim: 
$\theta_{S,T}  \circ \rho_{g,h} = \theta_{g(S),h(T)}$.  This is true since both sides are contractive homomorphism, and they 
agree on monomials $z^n w^m$.   Indeed 
$$\theta_{g(S),h(T)}(z^n w^m) = g(S)^n \; h(T)^m.$$
Viewing $A(\Ddb^2)$ as the closure of the tensor product of $A(\Ddb)$ with itself, we have $\theta_{S,T}(p(z) q(w)) = p(S) q(T)$ 
if $p, q$ are polynomials, and also if $p, q \in A(\Ddb)$ by approximating by polynomials.  Thus
$$\theta_{S,T}(\rho_{g,h}(z^n w^m) ) = \theta_{S,T}(g(z)^n h(w)^m) = g(S)^n h(T)^m .$$
From the Claim, if $f\in A(\Ddb^2)$, then we have $$f(g(S),h(T)) = \theta_{g(S),h(T)}(f) = 
\theta_{S,T}(\rho_{g,h}(f)) = \theta_{S,T}(f \circ (g,h)) = (f \circ (g, h))(S,T)$$
as desired.  
 \end{proof} 

\begin{lemma} \label{prodF}  \begin{itemize} \item [(1)]  We have $\{x^2 : x \in \fc \} = \{xy : x,y \in \fc \}$, and these coincide with the region $R$ inside the cardioid given by the polar equation $r = \frac{1}{2}\cos(\theta)+\frac{1}{2}$ for $\theta \in [-\pi,\pi].$ Hence if $x,y \in \fc $, then $x^\half y^\half  = (xy)^\half \in \fc$ also.
\item [(2)]  If
$b \in \frac{1}{2} {\mathfrak F}_{\Cdb}$, and if
$A$ is an 
operator algebra and $a \in \frac{1}{2} {\mathfrak F}_{A}$, then the numerical
range of $ab$ contains no strictly negative numbers, and the unique
accretive square root $(ab)^{\frac{1}{2}}$ is in $\frac{1}{2} {\mathfrak F}_{A}$.
\item [(3)]  
If  
$A$ is an 
operator algebra and $a, b \in \frac{1}{2} {\mathfrak F}_A$ with $ab = ba$ then $a^{\frac{1}{2}} b^{\frac{1}{2}}
 \in \frac{1}{2} {\mathfrak F}_A$. 
\item [(4)]
If
$A$ is an
operator algebra and $a, b \in {\mathfrak r}_A$ with $ab = ba$ then $a^{\frac{1}{2}} b^{\frac{1}{2}}
 \in {\mathfrak r}_A$.   \end{itemize}
    \end{lemma}

\begin{proof}  
(1) \    The boundary of $\fc$ is the circle given by the polar equation $r=\cos(\theta)$ for $\theta \in [-\frac{\pi}{2},\frac{\pi}{2}]$. So if $x \in \fc$, then $x=re^{i\theta}$ for some $\theta \in [-\frac{\pi}{2},\frac{\pi}{2}]$ and some $0 \leq r \leq \cos(\theta)$.  Thus $x^2=r^2e^{i2\theta}$, where $r^2 \leq \cos^2(\theta) = \half\cos(2\theta)+\half$. Hence $\{x^2 : x \in \fc \} \subseteq R$. For the other direction, if $se^{i\psi} \in R$, that is, $0 \leq s \leq \frac{1}{2}\cos(\psi)+\frac{1}{2}$ for some $\psi \in [-\pi,\pi]$, then $s^{\half} \leq ( \frac{\cos(\psi)+1}{2})^{\half}=\cos(\frac{\psi}{2})$, which shows that $s^{\half}e^{i\psi/2} \in \fc$. So we have shown $\{x^2 : x \in \fc \} = R$.

Since clearly $\{x^2 : x \in \fc \} \subseteq \{xy : x,y \in \fc \}$, it remains to show that $\{xy : x,y \in \fc \} \subseteq R$. To this end, suppose $x=re^{i\theta}$ and $y=se^{i\psi}$ for some $\theta, \psi \in [-\frac{\pi}{2},\frac{\pi}{2}]$, $0 \leq r \leq \cos(\theta)$, and $0 \leq s \leq \cos(\psi)$. We wish to show that $xy = rse^{i(\theta+\psi)} \in R$, or $0 \leq rs \leq \half \cos(\theta+\psi)+\half$. Since $rs \leq \cos(\theta)\cos(\psi)$, we will be done if $\cos(\theta)\cos(\psi) \leq \half\cos(\theta+\psi)+\half$. But this is clear since $\cos(\theta)\cos(\psi) =   \half\cos(\theta+\psi) + \half\cos(\theta-\psi)$.  

(3) \  Consider the function $f(z,w) = 1 - 2 (\frac{1-z}{2})^{\frac{1}{2}} (\frac{1-w}{2})^{\frac{1}{2}}$
on the bidisk $\bar{\Ddb} \times \bar{\Ddb}$.   By (1), $f$ takes values in $\bar{\Ddb}$, and it is
clearly a member of the bidisk algebra.    By the two-variable von Neumann inequality 
resulting from Ando's dilation theorem  (see e.g.\ 2.4.13 in \cite{BLM}), we have $\Vert f(1-2a,1-2b) \Vert \leq 1$.  We claim that 
$f(1-2a,1-2b) = 1 - 2 a^{\frac{1}{2}} b^{\frac{1}{2}}$.    To see this, let
$g(z) = h(z) =  ((1-z)/2)^{\frac{1}{2}}$, then $g(1-2a) = a^{\frac{1}{2}}$ and $h(1-2b) = b^{\frac{1}{2}}$ as in 
\cite[Proposition 2.3 ]{BR1}.
Letting $r(z,w) = 1 - 2 z w$, we have by Lemma \ref{vnfc}, with the $f$ there replaced by $r$,
that $1 - 2 a^{\frac{1}{2}} b^{\frac{1}{2}}
= f(1-2a,1-2b)$.  

(2) \  We may assume that $b \neq 0$.
For any state $\varphi$  of $A$ we have $\varphi(ab) = \varphi(a) b$,
and this is in the cardioid in (1), since $\varphi(a)$ and $b$ are in 
$\frac{1}{2} {\mathfrak F}_{\Cdb}$ (note that 
$|1 - 2 \varphi(a)| \leq \Vert 1 - 2 a \Vert$.  Thus the numerical
range of $ab$ contains no strictly negative numbers, and by
\cite{MP,LRS} a unique
accretive square root exists.    
The rest follows from (3), or may be proved directly using
a von Neumann inequality/disk algebra
functional calculus argument applied to the function 
$f(z) = 2 (b (z+1)/2)^{\frac{1}{2}} -1$ on the disk. 

(4) \ Suppose that $a, b \in {\mathfrak r}_A$ with $ab = ba$.
Then let $a_t = t a (1 + ta)^{-1}$ and $b_t = tb (1 + tb)^{-1}$ for $t > 0$.  These are in $\frac{1}{2}{\mathfrak F}_{A}$ as explained in
\cite[Section 3]{BR2}, and they commute, by algebra.  So $a_t^{\frac{1}{2}} b_t^{\frac{1}{2}}
\in \frac{1}{2}{\mathfrak F}_{A} \subset {\mathfrak r}_A$.  Therefore
    $$(\frac{1}{t} a_t)^{\frac{1}{2}} (\frac{1}{t} b_t)^{\frac{1}{2}}
= \frac{1}{t} a_t^{\frac{1}{2}} b_t^{\frac{1}{2}} \in {\mathfrak r}_A .$$
Taking the limit as $t \to 0$, and using the continuity of the roots stated in the introduction,
we see that  $a^{\frac{1}{2}} b^{\frac{1}{2}}
 \in {\mathfrak r}_A$.        \end{proof}

An application of the last result to  noncommutative Urysohn-type lemmas is given in \cite{BRIII}.

\medskip

\noindent {\em Remark.}   It is not true in general that $a^{\frac{1}{2}} b^{\frac{1}{2}} \in {\mathfrak r}_{A}$
for noncommuting $a, b \in \frac{1}{2}{\mathfrak F}_{A}$.   

\begin{corollary} \label{manr}  If  
$A$ is an 
operator algebra, $n \in \Ndb$, and $a_1, \cdots, a_n$ are in 
$\frac{1}{2} {\mathfrak F}_A$ (resp.\ in ${\mathfrak r}_{A}$) with $a_i a_j = a_j a_i$ for all
$i, j$, then $a_1^{\frac{1}{n}} \cdots \,  a_n^{\frac{1}{n}}$ is in 
$\frac{1}{2} {\mathfrak F}_A$ (resp.\ in ${\mathfrak r}_{A}$).  
 \end{corollary}  \begin{proof}   We just do the ${\mathfrak F}_A$ case, the other being similar.
 We first prove this if $n = 2^k$ by induction
on the integer $k$.   Assuming this is true for $n = 2^k$, write 
$$a_1^{\frac{1}{2n}} \cdots \,  a_{2n}^{\frac{1}{2n}} = 
(a_1^{\frac{1}{n}} \cdots  \, a_{n}^{\frac{1}{n}})^{\frac{1}{2}} \, (a_{n+1}^{\frac{1}{n}} \cdots  a_{2n}^{\frac{1}{n}})^{\frac{1}{2}} .$$
By the inductive hypothesis $a_1^{\frac{1}{n}} \cdots  \, a_{n}^{\frac{1}{n}} \in  \frac{1}{2} {\mathfrak F}_A$
and $a_{n+1}^{\frac{1}{n}} \cdots \,  a_{2n}^{\frac{1}{n}} \in  \frac{1}{2} {\mathfrak F}_A$.
Hence by Lemma \ref{prodF} we have  $(a_1^{\frac{1}{n}} \cdots \,  a_{n}^{\frac{1}{n}})^{\frac{1}{2}} \, (a_{n+1}^{\frac{1}{n}} \cdots \,  a_{2n}^{\frac{1}{n}})^{\frac{1}{2}}$ in $\frac{1}{2} {\mathfrak F}_A$.
This completes the induction step.

To see that the result holds for every $n \in \Ndb$, we just do the case $n = 3$ as an illustration.  So suppose that 
$x, y, z \in  \frac{1}{2} {\mathfrak F}_A$.
For a  large integer $k$ set $m = 2^k, p = [m/3]$, and set $a_k = x$ for $1 \leq k \leq p$, and $a_k = y$ for $p+1 \leq k \leq 2 p$,
and $a_k = z$ for $2p + 1 \leq k \leq m$.   Then 
$$(x^{\frac{1}{m}})^p \,   (y^{\frac{1}{m}})^p \, (z^{\frac{1}{m}})^{m-2p} =
a_1^{\frac{1}{m}} \cdots  \, a_{m}^{\frac{1}{m}} \in  \frac{1}{2} {\mathfrak F}_A .$$
However $x^{\frac{p}{m}} \to x^{\frac{1}{3}}$ as $m \to \infty$ by continuity of powers in 
the exponent variable (mentioned in the introduction), and similarly 
 $y^{\frac{p}{m}} \to y^{\frac{1}{3}}$ and  $z^{\frac{m-2p}{m}} \to z^{\frac{1}{3}}$.
Since $\frac{1}{2} {\mathfrak F}_A$ is closed we deduce that 
$x^{\frac{1}{3}} y^{\frac{1}{3}} z^{\frac{1}{3}}  \in  \frac{1}{2} {\mathfrak F}_A$.
\end{proof}

Some ideas used in the  following three results came from discussions with Charles Batty (see Acknowledgements below).  

\begin{corollary} \label{ch}  If  
$A$ is an 
operator algebra, $n \in \Ndb$, and $a_1, \cdots, a_n$ are in 
$\frac{1}{2} {\mathfrak F}_A$ (resp.\  ${\mathfrak r}_{A}$),  with $a_i a_j = a_j a_i$ for all
$i, j$, 
 then $a_1^{s_1} \, a_2^{s_2} \cdots a_n^{s_n}$ is in 
$\frac{1}{2} {\mathfrak F}_A$ (resp.\ in ${\mathfrak r}_{A}$) for positive 
$s_1, \cdots, s_n$ with $\sum_{k=1}^n \, s_k \leq 1$.  
 \end{corollary}  \begin{proof}   Again we just do the ${\mathfrak F}_A$ case, and  $n = 3$, the other cases being similar.
 So suppose that 
$x, y, z \in  \frac{1}{2} {\mathfrak F}_A$.   We may assume that $1 \in A$ by passing to $A^1$ if necessary.
The last result shows that $x^{\frac{m}{N}} \, y^{\frac{k}{N}} \, z^{\frac{p}{N}}  \cdot 1^{\frac{r}{N}}
\in \frac{1}{2} {\mathfrak F}_A$, 
for any positive integers $N, k, m, p, r$ with $k + m + p + r =  N$.
Taking limits, using  continuity of powers as in the last proof, gives the assertion. 
\end{proof}  

\begin{lemma} \label{numrange2}
If $x,y \in {\mathfrak r}_{A}$ with $xy=yx$, and with the numerical range  $W(x) \subseteq S_{\varphi}$ and $W(y) \subseteq S_{\psi}$,  for some $\varphi,\psi \in [0,\frac{\pi}{2}]$, then $W(x^{1/2}y^{1/2}) \subseteq S_{(\varphi+\psi)/2}$.
\end{lemma}

\begin{proof}  Set $\alpha = \frac{\pi}{2}-\varphi$ and $\beta = \frac{\pi}{2}-\psi \geq 0$, so that $$W(e^{i\alpha} x),W(e^{-i\alpha} x),W(e^{i\beta} y),W(e^{-i\beta} y) \subseteq S_{\pi/2}.$$ Then by Corollary  \ref{ch} or Lemma \ref{prodF} (4), we have  $$W((e^{i\alpha}x)^{1/2}(e^{i\beta}y)^{1/2}) = e^{i(\alpha+\beta)/2} W(x^{1/2}y^{1/2}) \subseteq S_{\pi/2},$$ and similarly $$e^{-i(\alpha+\beta)/2} W(x^{1/2}y^{1/2}) \subseteq S_{\pi/2}.$$ The last two displayed equations together imply $$W(x^{1/2}y^{1/2}) \subseteq S_{(\pi-(\alpha+\beta))/2} = S_{(\varphi+\psi)/2}$$ as desired.
\end{proof}  

The following is a stronger version of the assertion for ${\mathfrak r}_{A}$ in Corollary  \ref{ch}.  

\begin{corollary} \label{chne}  
 Suppose that $A$ is an
operator algebra, $n \in {\mathbb N}$, and $a_1, \cdots, a_n$ are in
$A$,  with $a_i a_j = a_j a_i$ for all
$i, j$.  Suppose further that $W(a_k)$ contains no strictly negative numbers
for all $k$,
and that $s_1, \cdots, s_n$ are positive scalars with $\sum_{k=1}^n \, s_k \leq 1$.
\begin{itemize}
\item [(1)]  Suppose that $0 \leq \varphi_k \leq \pi$, and that  $W(a_k) \subseteq S_{\varphi_k}$, when $1 \leq k \leq n$.
If $\varphi = \sum_{k=1}^n \, s_k  \, \varphi_k$, then
$W(a_1^{s_1} \, a_2^{s_2} \cdots a_n^{s_n}) \subseteq S_{\varphi}$.
\item [(2)]  There exist angles with   $|\theta_k| \leq \frac{\pi}{2}$,
and $0 \leq \varphi_k \leq \frac{\pi}{2}$, and
$W(a_k) \subset e^{i \theta_k} S_{\varphi_k}$, for $1 \leq k \leq n$.
In this case we have
$W(a_1^{s_1} \cdots  a_n^{s_n}) \subset e^{i \theta} S_{\varphi}$, where
$\theta = \sum_{k=1}^n s_k \theta_k$, and $\varphi = \sum_{k=1}^n s_k \varphi_k$.
 \end{itemize}
\end{corollary}  \begin{proof}    First, we assume that $\varphi_k \leq \frac{\pi}{2}$.  We just sketch the proof of (1)
 in this case, since it follows the structure of the proof of Corollary \ref{ch} and the results 
leading up to that.   Item (1) in the case that $n = 2^N$ and $s_k = \frac{1}{2^N}$ is proved by induction on $N \geq 1$, similarly to the first paragraph
of the proof of Corollary \ref{manr}; the case $N = 1$ being Lemma \ref{numrange2}.    The case for general $n$ and  $s_k = \frac{1}{n}$ then follows similarly to the second paragraph
of that proof.    Finally the case of (1) for general positive $s_k$ with $\sum_{k=1}^n \, s_k \leq 1$ follows from the case in the 
last line by the idea in the proof of  Corollary  \ref{ch}.

Next, we state a general fact about an element $a \in A$ with numerical range 
avoiding the strictly negative real axis.  Here $W(a)$, being convex, must lie on a closed half-plane with $0$ in its boundary, and hence we can rotate clockwise by an angle  $\theta$ with $|\theta|  \leq  \frac{\pi}{2}$, to obtain $e^{-i \theta} a$ accretive.  Claim: 
$(e^{-i \theta} a)^s = e^{-i s \theta} \, a^s$
whenever $0 \leq s \leq 1$.  (Note that  this is not obvious; even if $A = \Cdb$ one  has to beware 
of simple sounding `identities' about roots, because of the issue of the `principal root'.  E.g.\ $(w z)^{\frac{1}{2}} \neq w^{\frac{1}{2}} z^{\frac{1}{2}}$ for numbers in the third quadrant.)
To see this, suppose that $\theta \geq 0$ (the negative case is similar).  Then $W(e^{i (\frac{\pi}{2} -\theta)} a)$ lies in the closed upper half plane.
By the first few lines of the proof of \cite[Theorem 2.8]{LRS}, we know that 
$$a^{\frac{1}{n}} = (e^{i (\frac{\pi}{2} -\theta)} a)^{\frac{1}{n}} \cdot e^{-i (\frac{\pi}{2} -\theta)/n} , \qquad n \in \Ndb $$
and that the numbers in $W((e^{i (\frac{\pi}{2} -\theta)} \, a)^{\frac{1}{n}})$ have argument in $[0,\frac{\pi}{n}]$.  
So the numbers in 
$W(a^{\frac{1}{n}}) =  e^{-i (\frac{\pi}{2} -\theta)/n} \, W((e^{i (\frac{\pi}{2} -\theta)} \, a)^{\frac{1}{n}})$ have argument in $$[-\frac{1}{n}(\frac{\pi}{2} -\theta),
\frac{1}{n}(\frac{\pi}{2} + \theta)] \subset [-\frac{\pi}{n}, \frac{\pi}{n}].$$
By the uniqueness 
assertion in \cite[Theorem 2.8]{LRS}, these $n$th roots are the usual (principal) ones.  
We also deduce that the numbers in 
$W(e^{-i \theta/n} \, a^{\frac{1}{n}})  = e^{-i \frac{\theta}{n}} \, W(a^{\frac{1}{n}})$ have argument in $[-\frac{\pi}{n},  \frac{\pi}{n}]$, and so 
by  the uniqueness 
assertion in \cite[Theorem 2.8]{LRS} again, $(e^{-i \theta} a)^{\frac{1}{n}} = e^{-i \frac{\theta}{n}} \, a^{\frac{1}{n}}$.
Raising to the power of a positive integer $m \leq n$, we obtain the Claim when $s$ is rational.  Hence it holds for 
all $s \in [0,1]$ by the continuity of $a^s$ in $s$, which is well known (particularly in the accretive case, which may be  `rotated' to give the 
general case).

Finally, let  $a_1, \cdots, a_n$ be as in the full statement of the corollary.  As in the last paragraph, we rotate by an
angle  $\theta_k$ with $|\theta_k| \leq  \frac{\pi}{2}$, to obtain $W(e^{-i \theta_k} a_k) \subset S_{\varphi_k}$, where $0 \leq \varphi_k \leq \frac{\pi}{2}$.   Applying the case proved in the first paragraph of the proof, if $s_1, \cdots, s_n$ are positive scalars with $\sum_{k=1}^n \, s_k \leq 1$, and if $\theta = \sum_{k=1}^n \, s_k \theta_k$
and $\varphi = \sum_{k=1}^n \, s_k  \, \varphi_k$, 
then  $W(e^{-i \theta} \, a_1^{s_1} \, a_2^{s_2} \cdots a_n^{s_n}) \subseteq S_{\sum_{k=1}^n \, s_k  \, \varphi_k}$.  (Here  we are using the Claim proved in the last 
paragraph).   
So $W(a_1^{s_1} \, a_2^{s_2} \cdots a_n^{s_n}) \subseteq e^{i \theta} \, S_{\varphi}$.  This proves (2),
from which (1) follows as an easy exercise.
\end{proof}  

\noindent {\em Remark.}  A similar proof works for mutually commuting  $a_1, \cdots, a_n$, such that for each $k$ there is some ray starting at the origin which 
avoids $W(a_k)$.

\begin{proposition} \label{incroo}   If $A$ is an operator algebra and $x \in \frac{1}{2} {\mathfrak F}_A$ then
$({\rm  Re}(x^{\frac{1}{n}}))$ is increasing.  \end{proposition}  

  \begin{proof}   We will prove a little more.  Let $0 < s < t \leq 1$, and
write $$f(z) = ((1-z)/2)^{s} - ((1-z)/2)^{t} \; ,  \qquad z \in \Cdb, |z| \leq 1.$$
  This has positive real part by the easy  case  of the present 
result where $A = \Cdb$.
Then apply \cite[Proposition 3.1, Chapter IV]{NF} to deduce that 
$f(1-2x)$ is accretive.   Here $f(1-2x)$  is the `disk algebra functional calculus', arising from von Neumann's inequality for the contraction $1-2x$,
applied to $f$.  As in \cite[Proposition 2.3]{BR2} we have $f(1-2x) = x^{s} - x^t$.  
So Re$(x^{s} - x^t) \geq 0.$
\end{proof}

\noindent {\em Remarks.}  1) \ Proposition \ref{incroo} is false in general for norm $1$ elements of ${\mathfrak r}_A$.  A counterexample is
the normalization of the matrix with rows $1$ and $i$, and $i$ and $0$.  However it is shown in \cite{BRIII}
that for any $a \in {\mathfrak r}_A$ there is a positive constant $c$ with ${\rm  Re}((c a)^{\frac{1}{n}}))_{n \geq 2}$ increasing.

\smallskip

2)  \ 
Proposition \ref{incroo}  may be used to give the existence of `increasing' approximate  identities
in separable approximately unital operator algebras \cite{BRIII}.

\subsection*{Acknowledgments}
Sections 2 and 3 of the paper are from mid 2012, and originated  in discussions between the second and third authors.  The main results here were advertised in \cite{BR2} but without proofs.
Section 4 is from ongoing discussions between the first and second authors beginning mid 2013.
We thank Charles Batty for some helpful comments in answer to a question.
  In particular, Corollary  \ref{ch}  was pointed out to us by him in the case $n = 2$,
as was also the special case of the first assertion of
Corollary  \ref{chne} when $n = 1$ and $\varphi_k \leq \frac{\pi}{2}$ 
(of course here $k = 1$).  His proof of the latter contained one of the ingredients
we needed in Lemma \ref{numrange2}.

\end{document}